\documentclass[twoside,
]{amsart}

\usepackage{srcltx}
  \usepackage[utf8]{inputenc}
  \usepackage{amsmath}
  \usepackage{amssymb}
  \usepackage{amsthm}
  \usepackage{mathrsfs}
  \usepackage{bm}
  \usepackage{bbm}
  \usepackage{pifont}
  \usepackage{graphicx}
  \usepackage{xmpmulti}
  \usepackage{color}

	\newtheoremstyle{slanted}
	{}
	{}
	{\slshape}
	{}
	{\bfseries}
	{.}
	{ }
	{}
	
	\theoremstyle{slanted}
	\newtheorem{theo}{Theorem}[section]
	\newtheorem{prop}[theo]{Proposition}

	\newtheorem{lemma}[theo]{Lemma}

	\newtheorem{remark}[theo]{Remark}

	\def\egdef{:=}

	\newcommand{\tend}[3][]{\xrightarrow[#2\to#3]{#1}}

	\newcommand{\EE}{\mathbb{E}}
	\newcommand{\Var}{\text{Var}}
	\def\esp#1{{\EE}\left[#1\right]}
	\def\espc#1#2{{\EE}\left[#1\,\left|\,#2\right.\right]}
	\def\ind#1{\mathbbmss{1}_{#1}}
	\newcommand{\ZZ}{\mathbb{Z}}

	\newcommand{\RR}{\mathbb{R}}
	\newcommand{\PP}{\mathbb{P}}

\newcommand{\Pb}{\mu^{\text{B}}}
\newcommand{\mant}{\mathscr{M}}
\newcommand{\convURI}{\xrightarrow{\text{URI}}}
\newcommand{\E}{{\bm{E}}}
\newcommand{\X}{{\bm{X}}}

\title{Averaging along Uniform Random Integers}

\author{Élise Janvresse  and Thierry de la Rue}

\address{Laboratoire de Math\'ematiques Rapha\"el Salem\\
	UMR 6085 CNRS -- Universit\'e de Rouen\\
	Avenue de l'Universit\'e\\
	B.P. 12\\
	F76801 Saint-\'Etienne-du-Rouvray Cedex}
\email{Elise.Janvresse@univ-rouen.fr}  \email{Thierry.de-la-Rue@univ-rouen.fr}

\begin{document}
\bibliographystyle{amsplain}

\begin{abstract}
Motivated by giving a meaning to ``The probability that a random integer has initial digit $d$'', we define a \emph{URI-set} as a random set $\E$ of natural integers such that each $n\ge1$ belongs to $\E$ with probability $1/n$, independently of other integers. This enables us to introduce two notions of densities on natural numbers: The \emph{URI-density}, obtained by averaging along the elements of $\E$, and the \emph{local URI-density}, which we get by considering the $k$-th element of $\E$ and letting $k$ go to $\infty$. We prove that the elements of $\E$ satisfy Benford's law, both in the sense of URI-density and in the sense of local URI-density. Moreover, if $b_1$ and $b_2$ are two multiplicatively independent integers, then the mantissae of a natural number in base $b_1$ and in base $b_2$ are independent. Connections of URI-density and local URI-density with other well-known notions of densities are established: Both are stronger than the natural density, and URI-density is equivalent to $\log$-density. We also give a stochastic interpretation, in terms of URI-set, of the $H_\infty$-density.
\end{abstract}

\maketitle

\section{Introduction}
\subsection{Benford's law and Flehinger's theorem}
Benford's law describes the empirical distribution of the leading digit of everyday-life numbers. It was first discovered by the astronomer Simon Newcomb in 1881 \cite{newcomb} and named after the physicist Franck Benford who independently rediscovered the phenomenon in 1938 \cite{benford}. According to this law, the proportion of numbers in large series of empirical data with leading digit $d\in\{1,2,\ldots,9\}$ is $\log_{10}\left(1+1/d\right)$. More generally, defining the \emph{mantissa} $\mant(x)$ of a positive real number $x$ as the only real number in $[1,10[$ such that $x=\mant(x) 10^k$ for some integer $k$, Benford's law states that for any $1\le \alpha<\beta<10$, the proportion of numbers whose mantissa lies in $[\alpha,\beta]$ is $\log_{10}\beta-\log_{10}\alpha$. 

Giving Benford's law a mathematical meaning requires to formalize the notion of ``everyday-life numbers'', which is far from obvious. However there have been many attempts to explain mathematically the ubiquity of this distribution in empirical datasets. One of them is Betty J. Flehinger's theorem, published in 1965 in an article entitled \emph{On the probability that a random integer has initial digit $A$} \cite{flehinger}.
It occurred to Flehinger that the most natural set of numbers on which we should verify Benford's distribution is the whole set of positive integers. Unfortunately, defining
$$ P_n^1(d) \egdef \dfrac{1}{n} \sum_{j=1}^n \ind{[d,d+1[}\bigl(\mant(j)\bigr) $$
(that is the proportion of integers between 1 and $n$ with leading digit $d$), we see that the sequence $P_n^1(d)$ has no limit as $n\to\infty$: It oscillates over longer and longer periods. Flehinger's idea was then to seek the limit by iteration of the process of Cesaro averaging: She inductively set
$$  P_n^{k+1}(d) \egdef \dfrac{1}{n} \sum_{j=1}^nP_j^k(d), $$
and proved that
\begin{equation}
 \label{eq:Flehinger's theorem}
\lim_{k\to\infty}\limsup_{n\to\infty} P_n^k(d) = \lim_{k\to\infty}\liminf_{n\to\infty} P_n^k(d) = \log_{10} \left(1+\dfrac{1}{d}\right),
\end{equation}
which is the proportion predicted by Benford's law. (Donald Knuth generalized Flehinger's theorem to the distribution of the whole mantissa in 1981 \cite{knuth}.)

\medskip

In spite of its title, Flehinger's article has no probabilistic content. A good reason is that there is no way of picking an integer uniformly at random in the set of all natural numbers. The first motivation for the present work was nevertheless to translate Flehinger's theorem in the context of probability theory: How can we interpret the \emph{probability} that a (random) integer has a given initial digit? Our purpose is thus to give a meaning to the sentence ``An integer picked uniformly at random has such a property''.

\subsection{Roadmap of the paper}
We construct in Section~\ref{sec:URI-set} a random infinite set $\E$ of integers, which we call a \emph{URI-set}, such that averaging along the elements of $\E$ reflects the expected behaviour of a random integer. 
This random set enables us to introduce two notions of densities on natural numbers: The \emph{URI-density} (see Section~\ref{sec:URI-density}), obtained by averaging along the elements of $\E$, and the \emph{local URI-density} (see Section~\ref{sec:localURI-density}), which we get by considering the $k$-th element of $\E$ and letting $k$ go to $\infty$. We prove that the elements of $\E$ satisfy Benford's law, both in the sense of URI-density (Theorem~\ref{thm:benford}), and in the sense of local URI-density (Theorem~\ref{thm:single}). Our point of view also enables to consider simultaneously the mantissae of a number in different bases, and in particular we prove a result which can be interpreted as follows: If $b_1$ and $b_2$ are two multiplicatively independent integers, then the mantissae of a natural number in base $b_1$ and in base $b_2$ are independent (Theorem~\ref{thm:base}). Connections of URI-density and local URI-density with other well-known notions of densities are established: We prove that both are stronger than the natural density (Theorems~\ref{thm:cesaro} and~\ref{thm:natural_implies_local_URI}), and that in fact URI-density is equivalent to $\log$-density (Theorem~\ref{thm:log-density}). We finish in Section~\ref{sec:discussion} by giving a stochastic interpretation, in terms of URI-set, of the $H_\infty$-density used in Flehinger's theorem, and by raising some open problems.

 The construction of the random set of integers is inspired by a previous article by the same authors~\cite{janvresse}, where a probabilistic proof of Flehinger's theorem was provided. We summarize this proof in Section~\ref{sec:Markov}

\subsection*{Acknowledgements}
The authors are grateful to Bodo Volkmann for stimulating questions.

\section{Uniform random set of integers}
\label{sec:URI-set}

\subsection{Flehinger's theorem through Markov chain}
\label{sec:Markov}

We introduce a homogeneous Markov chain $(M_k)_{k\ge 0}$ taking values in $[1,10[$, defined by its initial value $M_0$ (which can be deterministic or random) and the following transition probability: for any Borel set $S\subset[1,10[$,
\begin{equation}
 \label{eq:transition probability}
\PP \bigl(M_{k+1}\in S \big| M_k=a\bigr) \egdef \PP \Bigl(\mant(aU)\in S\Bigr),
\end{equation}
where $U$ is a uniform random variable in $[0,1]$. 

Let us denote by $\Pb$ the probability distribution on $[1,10[$ given by Benford's law: For any $1\le t<10$,
$$ \Pb([1,t]) \egdef \log_{10} t. $$
It is proved in~\cite{janvresse} by a standard coupling argument that
\begin{itemize}
    \item $\Pb$ is the only probability distribution on $[1,10[$ which is invariant under the probability transition~\eqref{eq:transition probability};
    \item Whatever choice we make for the initial condition $M_0$, we have for any Borel set $S\subset[1,10[$ and for all $k\ge1$
\begin{equation}
 \label{eq:speed of convergence}
 \bigl| \PP(M_k\in S) - \Pb(S) \bigr| \le \left(\dfrac{9}{10}\right)^k.
\end{equation}
\end{itemize}
A connection is made between the quantities $P_n^k(d)$, $n\ge1$, and the $k$-th step of our Markov chain: It is established in~\cite{janvresse} that for all $a\in[1,10[$ and all $k\ge1$,
$$ \lim_{j\to\infty} P^k_{\lfloor a 10^j \rfloor} (d) = \PP \Bigl(M_{k}\in [d,d+1[ \,\big|\, M_0=a\Bigr). $$
A proof of~\eqref{eq:Flehinger's theorem}, with an estimation of the speed of convergence, follows:
We get for all $k\ge1$
\begin{equation}
 \begin{split}
&\Bigl| \liminf_{n\to\infty} P_n^k(d) - \Pb\Bigl([d,d+1[\Bigr) \Bigr| \le \left(\dfrac{9}{10}\right)^k \\
\text{and}\quad
&\Bigl| \limsup_{n\to\infty} P_n^k(d) - \Pb\Bigl([d,d+1[\Bigr) \Bigr| \le \left(\dfrac{9}{10}\right)^k.
\end{split}
\end{equation}

\subsection{Construction of a random set of integers}
We can interpret the Markov chain  $(M_k)$ as the sequence of mantissae of positive random variables $X_k$, where the sequence $(X_k)$ is itself a Markov chain such that, given $X_0,\ldots,X_k$, $X_{k+1}$ is uniformly distributed in $]0,X_k[$. Then $\X\egdef\{X_k,\ k\ge0\}$ is a discrete random set of real numbers which satisfies the following property: For any $t>0$, conditionally to the fact that $\X\cap]t,\infty[\neq\emptyset$, $\max(\X\cap]0,t])$ is uniformly distributed in $]0,t]$, and independent of $\X\cap]t,\infty[$.  

Our idea is thus to imitate the structure of this random set of reals, but inside the set of natural numbers. 
We are looking for a random infinite set of integers $\E$ satisfying
\begin{enumerate}
 \item[(U)] for all $n\ge1$, $\max (\E\cap\{1,\ldots,n\})$ is uniformly distributed in $\{1,\ldots,n\}$;
 \item[(I)] for all $n\ge1$, $\max (\E\cap\{1,\ldots,n\})$ is independent of $\E\cap\{n+1,n+2,\ldots\}$.
\end{enumerate}
For such a random set $\E$, we must have by (U), for each $n\ge1$, 
$$\PP(n\in\E)=\PP\bigl(\max (\E\cap\{1,\ldots,n\})=n\bigr)=1/n,$$ 
and (I) implies that all events $(n\in\E)$ are independent.

Conversely, picking elements of $\E$ using independent Bernoulli random variables, with $\PP(n\in\E)=1/n$ for each $n\ge1$ gives a random set satisfying the required conditions. Indeed, for each $j\in\{1,\ldots,n\}$, we get
\begin{multline*}
 \PP\Bigl(\max (\E\cap\{1,\ldots,n\})=j\Bigr) = \PP\Bigl(j\in \E, j+1\notin\E,\ldots,n\notin\E\Bigl) \\
= \dfrac{1}{j} \dfrac{j}{j+1}\cdots \dfrac{n-1}{n} = \dfrac{1}{n}.
\end{multline*}
Observe also that, with probability 1, the cardinality of $\E$ is infinite.

Because of the uniformity property (U), such a random set $\E$ appears as a good way to modelize the uniform distribution in the set of natural numbers and will therefore be referred to as a set of \emph{uniform random integers}, or \emph{URI-set}.

\section{URI-density and Benford's law}
{From} now on, $\E$ denotes a URI-set, and we denote its ordered elements by 
$$ \E = \{N_1=1<N_2<\ldots<N_k<\ldots\}.$$
For each $n\ge1$, we set $\E_n\egdef \E\cap\{1,\ldots,n\}$.

It will be useful to give the following estimation of $|\E_n|$.

\begin{lemma}
 \label{lemma:E_n}
$$ \dfrac{|\E_n|}{\ln n}\tend[a.s.]{n}{\infty} 1. $$
\end{lemma}

We recall Theorem~12 page 272 in Petrov's book~\cite{petrov}: 

\begin{theo}
 \label{thm:Petrov}
Let $(Z_n)$ be a sequence of independent centered real-valued random variables $Z_n$. If $t_n\uparrow\infty$ and
$$\sum_{n\ge 1} \dfrac{\esp{Z_n^p}}{t_n^p} <\infty ,$$
for some $p$, $1\le p\le 2$, then 
$$\dfrac{\sum_{j=1}^n Z_j}{t_j} \tend[a.s.]{n}{\infty} 0. $$
\end{theo}

\begin{proof}[Proof of Lemma~\ref{lemma:E_n}]
We consider the independent centered random variables $Z_n\egdef \ind{\E}(n)-1/n$. Since 
$$
\sum_{n\ge 1} \dfrac{\esp{Z_n^2}}{(\ln n)^2} 
\le \sum_{n\ge 1} \left(1-\dfrac{1}{n}\right)\dfrac{1}{n(\ln n)^2} <\infty ,
$$
we get by Theorem~\ref{thm:Petrov}
\begin{equation}
\label{eq:lnN}
 \dfrac{\sum_{j=1}^n Z_j}{\ln n} = \dfrac{1}{\ln n}\sum_{j=1}^n \ind{\E}(j)- \dfrac{1}{\ln n}\sum_{j=1}^n 1 /j \tend[a.s.]{n}{\infty} 0\, .
\end{equation}
This concludes the proof of the lemma.
\end{proof}

\subsection{URI-density}
\label{sec:URI-density}
We say that a subset $A$ of the set of natural numbers has \emph{URI-density} $\alpha$ if 
$$ \dfrac{1}{n} \sum_{k=1}^n\ind{A}(N_k) \tend[a.s.]{n}{\infty} \alpha. $$
Note that an equivalent formulation is
$$ \dfrac{1}{|\E_n|} \sum_{j\in\E_n} \ind{A}(j) \tend[a.s.]{n}{\infty} \alpha. $$


As we expect, the URI-density generalizes the notion of natural density.
\begin{theo}
\label{thm:cesaro}
Let $A\subset\ZZ_+$. If $\dfrac{1}{n}\sum_{j=1}^n \ind{A}(j)\tend{n}{\infty}\alpha$, then $A$ has URI-density~$\alpha$.
\end{theo}

\begin{proof}
When considering the elements of $\E_n$, it will be convenient to order them backwards:
$$ \E_n=\E\cap\{1,\ldots, n\} = \left\{Y_1^{(n)} > Y_2^{(n)} > \ldots > Y_{|\E_n|}^{(n)}=1\right\}. $$
For each $n\ge1$, let $a_n\egdef \ind{A}(n)$. 
We are going to prove the result in the form
\begin{equation}
 \label{eq:uri2}
\dfrac{1}{|\E_n|}\sum_{j\in \E_n}a_j \tend[a.s.]{n}{\infty} \alpha.
\end{equation}
We split the above average as 
\begin{equation}
\label{eq:split}
 \dfrac{1}{|\E_n|} \sum_{j\in \E_n}a_j 
= \dfrac{1}{|\E_n|}\sum_{i=1}^{|\E_n|} \left(a_{Y_i^{(n)}} - K_i^n\right) +\dfrac{1}{|\E_n|} \sum_{i=1}^{|\E_n|} K_i^n, 
\end{equation}
where 
$$ K_i^n \egdef \begin{cases}
                 \espc{a_{Y_i^{(n)}}}{Y_{i-1}^{(n)},\ldots,Y_1^{(n)}} & (2\le i\le |\E_n|),\\
		 \esp{a_{Y_1^{(n)}}} & (i=1).
                \end{cases}
$$
We first deal with the second term of~\eqref{eq:split}.
By Properties~(U) and~(I), $Y_1^{(n)}$ is uniformly distributed in $\{1,\ldots, n\}$, 
and conditionally to $\left(Y_{i-1}^{(n)},\ldots,Y_1^{(n)}\right)$, $Y_{i}^{(n)}$ is uniformly distributed 
in $\left\{1,\ldots,Y_{i-1}^{(n)}-1\right\}$, as long as $Y_{i-1}^{(n)}>1$.
Hence,
$$ 
K_1^n = \dfrac{1}{n} \sum_{j=1}^n a_j  \quad  \text{ and }\quad 
K_i^n = \dfrac{1}{Y_{i-1}^{(n)}-1} \sum_{j=1}^{Y_{i-1}^{(n)}-1} a_j,\quad (2\le i\le |\E_n|).
$$
By hypothesis, as soon as $Y_{i-1}^{(n)}$ is large, $K_i^n$ is close to $\alpha$. 
For any fixed $\varepsilon >0$, the number of $i$'s such that $|K_i^n-\alpha|>\varepsilon$ is bounded independently of $n$. 
Since $|\E_n|\to\infty$ a.s., it follows that 
$$\dfrac{1}{|\E_n|}\sum_{i=1}^{|\E_n|} K_i^n\tend[a.s.]{n}{\infty} \alpha\, . $$
Let us now turn to the first term of~\eqref{eq:split}.
Define for any $n\ge1$ and $i\ge1$
$$ A_i^n \egdef \begin{cases}
                 a_{Y_i^{(n)}} - K_i^n & (1\le i\le |\E_n|),\\
		 0 & (i> |\E_n|).
                \end{cases}
$$
In order to prove~\eqref{eq:uri2}, it remains to show that 
\begin{equation}
\label{eq:A}\dfrac{1}{|\E_n|}\sum_{i=1}^{|\E_n|} A_i^n \tend[a.s.]{n}{\infty} 0.
\end{equation}
Using a standard method, we first prove that~\eqref{eq:A} holds along a subsequence, then we control the oscillations to conclude.
Consider $n_m\egdef\lfloor\exp(m^2)\rfloor$.
By lemma~\ref{lemma:E_n}, convergence along the subsequence $(n_m)$ amounts to 
$$
A(n_m)\egdef\dfrac{1}{\lfloor\ln n_m\rfloor}\sum_{i=1}^{\lfloor\ln n_m\rfloor} A_i^{n_m} \tend[a.s.]{m}{\infty} 0\, .
$$
Observe now that the variance of $A_i^{n_m}$ is bounded, 
and that for $i\not= j$, $\esp{A_i^{n_m}A_j^{n_m}}=0$. 
Therefore, the variance of $A(n_m)$ is of order $(\ln n_m)^{-1}=1/m^2$.
By Tchebychev's inequality, 
$$
\sum_{m\ge1} \PP \left( A(n_m)>m^{-1/4} \right) 
\le \sum_{m\ge1} \frac{\Var(A(n_m))}{m^{-1/2}} 
= \sum_{m\ge1} m^{1/2}O\left(\frac{1}{m^2}\right) 
<\infty.$$
Hence, by Borel-Cantelli, with probability 1 we have $A(n_m)\le m^{-1/4}$ for $m$ large enough.
This proves that 
$$
\dfrac{1}{|\E_{n_m}|}\sum_{i=1}^{|\E_{n_m}|} A_i^{n_m} \tend[a.s.]{m}{\infty} 0.
$$

Consider now an integer $n\in]n_m, n_{m+1}[$. We can write
\begin{multline}
\label{eq:nmn}
\dfrac{1}{|\E_n|}\sum_{i=1}^{|\E_n|} A_i^n = 
\dfrac{1}{|\E_n|} \left( \sum_{i=1}^{|\E_n|} A_i^n - \sum_{i=1}^{|\E_{n_m}|}A_i^{n_m}\right)\\
+ \left\{ \dfrac{1}{|\E_{n}|} - \dfrac{1}{|\E_{n_m}|}  \right\} \sum_{i=1}^{|\E_{n_m}|}A_i^{n_m}
+  \dfrac{1}{|\E_{n_m}|}\sum_{i=1}^{|\E_{n_m}|} A_i^{n_m}
\end{multline}
Since $A_i^{n_m}$ is bounded and $|\E_{n_m}|-|\E_{n}| = o(|\E_{n}|)$, the second term on the RHS vanishes as $m\to\infty$.
Moreover, $\E_{n_m}\subset \E_n$. Therefore each $A_i^{n_m}$ in the first term (except $A_1^{n_m}$ which has a slightly different definition) is annihilated by some $A_j^n$, and the first term of~\eqref{eq:nmn} reduces to 
$$\dfrac{1}{|\E_n|}  \sum_{i=1}^{|\E_n|-|\E_{n_m}|+1} A_i^n - \dfrac{1}{|\E_n|} A_1^{n_m} ,$$
which goes to zero as $m\to\infty$. Since we already know that the third term goes to zero, this proves~\eqref{eq:A}.
\end{proof}

\subsection{Benford's law}
\label{sec:Benford}
We say that a sequence of positive real numbers $(x_n)$ follows Benford's law if for all $1\le t<10$
$$ \dfrac{1}{n}\sum_{j=1}^n \ind{\mant(x_j)<t} \tend{n}{\infty} \log_{10} t. $$
\begin{remark}
 \label{rem:inverse}
We recall that this is equivalent to the uniform distribution  $\mod 1$ of the sequence $(\log_{10} x_k)$ (see e.g. \cite{diaconis}). Therefore, this is also equivalent to the fact that the sequence $(1/x_k)$ follows Benford's law.
\end{remark}

The following theorem shows that the elements of the URI-set $\E$ almost surely follow Benford's law.

\begin{theo}
\label{thm:benford}
For any $1\le t\le 10$, the URI-density of $\{n\ge1:\ \mant(n)<t\}$ is $\log_{10} t$.
\end{theo}

\begin{proof}
By Duncan's work~\cite{duncan}, this result can be viewed as a corollary of the equivalence between URI-density and log-density (see Theorem~\ref{thm:log-density} below). However we provide a direct proof of it, in which useful ideas will be presented.

It is convenient to consider a coupling of the URI-set $\E$ and its continuous analog defined as follows:
Let $\xi$ be a Poisson process on $\RR_+^*$ with intensity $1/x$. It can be viewed as a random set of points.
For any interval $I$, let $\xi_I$ denote the number of points in $I\cap \xi$: $\xi_I$ is Poisson distributed with parameter $\int_I \frac{dx}{x}$.
{F}rom $\xi$, define the random set $\E$ as the set of integers $n\ge1$ such that $\xi_{]n-1, n]}\ge1$. 
Since the random variables $(\xi_{]n-1, n]})_{n\ge1}$ are independent and 
$$
\PP(n\in \E) = 1-\PP\left(\xi_{]n-1, n]}=0 \right) = \frac{1}{n}, 
$$
$\E$ is a URI-set.

The process $\xi$ satisfies a property analogous to Property~(U): 
For any $a>0$, the largest point of $[0, a]\cap\xi$ is uniformly distributed in $[0, a]$. 
Indeed, for any $t<a$
$$
\PP\Bigl( \max ([0, a]\cap\xi) \le t \Bigr)=\PP\left( \xi_{[t,a]} = 0 \right) = \frac{t}{a}.
$$
Let us order backwards the points of $\xi\cap [0,1]$: $1>Y_1>Y_2>\ldots $. 
Conditionally to $Y_k$, $Y_{k+1}$ is uniformly distributed in $[0, Y_k]$. 
By Proposition~3.1 of~\cite{janvresse}, the mantissae $(\mant(Y_k))$ constitute a Markov chain whose unique invariant distribution is $\Pb$. Moreover, $Y_1$ being uniformly distributed in $[0,1]$, the distribution of $\mant(Y_1)$ is the normalized Lebesgue measure on $[1,10[$, hence is equivalent to $\Pb$. By the pointwise ergodic theorem, we have for any $1\le t<10$
$$
\frac{1}{n}\sum_{k=1}^n \ind{\mant(Y_k)<t} \tend[a.s.]{n}{\infty} \log_{10} t.
$$
Consider now the points of $\xi\cap ]1, +\infty[$: $X_1<X_2<\ldots$
They follow the same distribution as $(1/Y_1, 1/Y_2, \ldots)$. 
By Remark~\ref{rem:inverse}, we deduce that 
\begin{equation}
 \label{eq:Benford for X}
\frac{1}{n}\sum_{k=1}^n \ind{\mant(X_k)<t} \tend[a.s.]{n}{\infty} \log_{10} t.
\end{equation}
Now observe that 
$$
\sum_{n\ge1}\PP (\xi_{]n-1, n]}\ge 2)=
\sum_{n\ge1}\left(\dfrac{1}{n}+\dfrac{n-1}{n}\log\dfrac{n-1}{n}\right)
< \infty,
$$
hence by Borel-Cantelli, with probability one there is only a finite number of $n$'s such that $\xi_{]n-1, n]}\ge 2$. Coming back to the URI-set $\E=\{N_1=1<N_2<\cdots\}$, this implies the almost-sure existence of $R$ such that, for all large enough $k$, 
$ 0\le N_{k}-X_{k+R}<1$. Since $N_k\to\infty$ a.s., we have $|\mant(N_k)-\mant(X_{k+R})|\to 0$ unless $N_k$ be of the form $10^p$ (which, again by Borel-Cantelli, happens almost surely for only finitely many $k$'s). It follows from~\eqref{eq:Benford for X} that
$$\frac{1}{n}\sum_{k=1}^n \ind{\mant(N_k)<t} \tend[a.s.]{n}{\infty} \log_{10} t.$$
\end{proof}

\subsection{Equivalence with log-density}
\label{sec:log}
To deal with the problem of non-existence of natural densities, several alternative densities have been introduced. Flehinger's theorem amounts to considering the so-called \emph{$H^\infty$-density}, obtained by iteration of Cesaro averages: A subset $A$ of $\ZZ_+$ has \emph{$H^\infty$-density} $\alpha$ if
$$
\lim_{k\to\infty}\limsup_{n\to\infty} P_n^k = \lim_{k\to\infty}\liminf_{n\to\infty} P_n^k = \alpha,
$$
where the $P_N^k$'s are inductively defined by $P_n^0\egdef\ind{A}(n)$ and $P_n^{k+1}\egdef (1/n) \sum_{j=1}^n P_j^k$. Obviously, $H^\infty$-density is \emph{stronger} than natural density, in the sense used by Diaconis in~\cite{diaconis_thesis}: If $A$ has natural density $\alpha$, then $A$ has $H^\infty$-density $\alpha$. The example of the set $A$ of integer whose initial digit (when written in base 10) is 1 shows that $H^\infty$-density is \emph{strictly} stronger than natural density.

Still stronger than $H^\infty$-density is the notion of \emph{log-density}.
Recall that $A\subset\ZZ_+$ has \emph{log-density} $\alpha$ if
$$\dfrac{1}{\ln n}\sum_{j=1}^n \dfrac{1}{j} \ind{A}(j) \tend{n}{\infty} \alpha.$$
A proof that log-density is stronger than $H^\infty$-density can be found in~\cite{diaconis}, together with an example of a set $A$ with a log-density, but for which the $H^\infty$-density fails to exist.

\begin{theo}
\label{thm:log-density}
 Let $A\subset\ZZ_+$. 
$$\dfrac{1}{|\E_n|}\sum_{j=1}^n \ind{A}(j) \ind{\E}(j) - \dfrac{1}{\ln n}\sum_{j=1}^n \ind{A}(j) /j \tend[a.s.]{n}{\infty} 0\, .$$
In particular, $A$ has URI-density $\alpha$  if and only if $A$ has log-density $\alpha$.
\end{theo}

\begin{proof}
We apply Theorem~\ref{thm:Petrov}: 
We consider the independent random variables $Z_n\egdef \ind{A}(n)(\ind{\E}(n)-1/n)$ which are centered. Since 
$$
\sum_{n\ge 1} \dfrac{\esp{Z_n^2}}{(\ln n)^2} \le  \sum_{n\ge 1} \left(1-\dfrac{1}{n}\right)\dfrac{1}{n(\ln n)^2} <\infty ,
$$
we get 
\begin{equation}
\label{eq:lnNbis}
 \dfrac{\sum_{j=1}^n Z_j}{\ln n} = \dfrac{1}{\ln n}\sum_{j=1}^n \ind{A}(j) \ind{\E}(j) - \dfrac{1}{\ln n}\sum_{j=1}^n \ind{A}(j) /j \tend[a.s.]{n}{\infty} 0\, .
\end{equation}
Since, by Lemma~\ref{lemma:E_n}
$$ \dfrac{|\E_n|}{\ln n}\tend[a.s.]{n}{\infty} 1, $$
we can conclude the proof of Theorem~\ref{thm:log-density}.
\end{proof}

\section{Independence of mantissae in different bases}
\label{sec:independent bases}
All the previous results concerned numbers written in the base-10 numeral system, but extend straightforwardlly to any integer base $b$. 
We denote by $\mant_b(x)\in[1, b[$ the \emph{mantissa in base $b$} of a positive real number $x$ and by $\Pb_b$ the probability distribution over $[1, b[$ defined by $\Pb_b([1, t]) \egdef \log_b t$ for all $1\le t <b$.

The purpose of this section is to prove the following theorem, which states that the mantissae in different bases of the elements of the URI-set $\E$ are independent, under some algebraic condition on the bases.

\begin{theo}
 \label{thm:base}
Let $(b_i)_{1\le i\le \ell}$ be positive integers, satisfying
\begin{equation}
 \label{eq:algebraic independence of logarithms}
 \forall a_1,\ldots,a_\ell\in\ZZ, \Bigl[\frac{a_1}{\ln b_1} +\cdots+\frac{a_\ell}{\ln b_\ell}=0\Bigr]\Longrightarrow
a_1=\cdots=a_\ell=0.
\end{equation}
Then 
for any $1\le t_i < b_i$ ($1\le i\le \ell$), 
$\left\{n\in\ZZ_+: \mant_{b_i}(n) \le t_i \ \forall 1\le i\le \ell \right\}$ has URI-density equal to 
$\prod_{i=1}^\ell\log_{b_{i}}t_i$.
\end{theo}

Recall that the positive integers $(b_i)_{1\le i\le \ell}$ are said to be \emph{multiplicatively independent} if $b_1^{s_1}\dots b_\ell^{s_\ell} =1$ where $(s_i)_{1\le i\le \ell} \subset \ZZ$ implies that $s_i=0$ for all $i$.
Note that, in the case $\ell=2$, property~\eqref{eq:algebraic independence of logarithms} exactly means that $b_1$ and $b_2$ are multiplicatively independent. To our knowledge, it is unknown whether, in the general case, multiplicative independence of $b_1,\ldots,b_\ell$ implies property~\eqref{eq:algebraic independence of logarithms}. This question is related to the so-called \emph{Schanuel's conjecture} in transcendental number theory (see \cite{Lang}, p.~30-31 or \cite{Waldschmidt}).

\begin{lemma}
 \label{lemma:random walk}
Let $(Z_k)_{k\ge1}$ be i.i.d. random variables taking values in $(\RR/\ZZ)^\ell$ with common distribution $\nu$. Assume that the only probability distribution $\mu$ such that $\mu\ast\nu=\mu$ is the Lebesgue measure on $(\RR/\ZZ)^\ell$. Then the random walk $(P_k)_{k\ge1}\egdef(Z_1+\cdots+Z_k)_{k\ge1}$ is uniformly distributed on $(\RR/\ZZ)^\ell$. In other words, it satisfies:
For all cylinder $C=[u_1,v_1]\times\cdots\times[u_\ell,v_\ell]$ where $0\le u_i<v_i<1$, 
\begin{equation}
 \label{eq:uniform distribution of random walk}
\dfrac{1}{n}\sum_{k=1}^n \ind{C}(P_k) \tend[a.s.]{n}{\infty} \prod_{i=1}^\ell (v_i-u_i). 
\end{equation}
\end{lemma}

\begin{proof}
 Denote by $\mu$ the Lebesgue measure on $(\RR/\ZZ)^\ell$. 
Let $M_0$ be a random variable with law $\mu$, independent of $(Z_k)_{k\ge1}$. Setting 
$$ M_k\egdef M_0+Z_1+\cdots+Z_k = M_0+P_k, $$
we get a stationary random walk $(M_k)_{k\ge0}$. Since $\mu$ is the unique invariant measure under convolution by $\nu$, the stationary process $(M_k)_{k\ge0}$ is ergodic, and by Birkhoff ergodic theorem, we get that for all cylinder $C=[u_1,v_1]\times\cdots\times[u_\ell,v_\ell]$ where $0\le u_i<v_i<1$, 
$$ \dfrac{1}{n}\sum_{k=1}^n \ind{C}(M_k) \tend[a.s.]{n}{\infty} \prod_{i=1}^\ell (v_i-u_i).$$
(See \emph{e.g.} \cite{Hernandez-Lasserre}, Corollary~2.5.2 page 38.)
Therefore, we can find some $m_0\in(\RR/\ZZ)^\ell$ such that, with probability~1, for all cylinder $C=[u_1,v_1]\times\cdots\times[u_\ell,v_\ell]$ where $0\le u_i<v_i<1$ are rational numbers,
$$ \dfrac{1}{n}\sum_{k=1}^n \ind{C}(m_0+P_k) \tend{n}{\infty} \prod_{i=1}^\ell (v_i-u_i).$$
We thus obtain that, for all cylinder $C$ with rational endpoints,
$$ \dfrac{1}{n}\sum_{k=1}^n \ind{C-m_0}(P_k) \tend[a.s.]{n}{\infty} \prod_{i=1}^\ell (v_i-u_i)=\mu(C-m_0).$$
By density of the rationals, \eqref{eq:uniform distribution of random walk} is satisfied for any cylinder $C$.
\end{proof}

\begin{proof}[Proof of Theorem~\ref{thm:base}]
 As in the proof of Theorem~\ref{thm:benford}, we consider the coupling of the URI-set with the Poisson process $\xi$, and we denote by $\cdots>X_2>X_1>1>Y_1>Y_2>\cdots$ the points of $\xi$. Define, for all $k\ge1$, $U_k\egdef Y_k/Y_{k-1}$ (where $Y_0\egdef 1$): Then $(U_k)_{k\ge1}$ is a sequence of i.i.d. uniform random variables in $[0,1]$, and $Y_k=U_1U_2\dots U_k$. 

Set $Z_k\egdef \Bigl( \log_{b_1}(U_k)\mod 1,\ldots, \log_{b_\ell}(U_k)\mod 1 \Bigr)\in(\RR/\ZZ)^\ell$, and let $\nu$ be the common law of the $Z_k$'s. We claim that the only probability measure $\mu$ which is invariant under convolution by $\nu$ is the Lebesgue measure on $(\RR/\ZZ)^\ell$. Indeed, for such an invariant measure,
the Fourier coefficients must satisfy
$$ \forall (m_1,\ldots,m_\ell)\in\ZZ^\ell,\quad \widehat\mu(m_1,\ldots,m_\ell) = \widehat\mu(m_1,\ldots,m_\ell) \widehat\nu(m_1,\ldots,m_\ell). $$
We just have to check that $\widehat\nu(m_1,\ldots,m_\ell)\neq1$ when $(m_1,\ldots,m_\ell)\neq(0,\ldots,0)$.
\begin{align*}\widehat\nu(m_1,\ldots,m_\ell) 
 &= \int_{(\RR/\ZZ)^\ell} e^{-i2\pi(m_1t_1+\cdots+m_\ell t_\ell)}\, d\nu(t_1,\ldots, t_\ell)\\
 &= \int_{[0,1]} e^{-i2\pi \left(m_1\log_{b_1} u + \cdots + m_\ell \log_{b_\ell} u\right)}\, du \\
 &= \int_{[0,1]} e^{-i2\pi\theta\ln u}\, du,
\end{align*}
where $\theta\egdef \frac{m_1}{\ln b_1} + \cdots + \frac{m_\ell}{\ln b_\ell}\neq0$ for $(m_1,\ldots,m_\ell)\neq(0,\ldots,0)$ by~\eqref{eq:algebraic independence of logarithms}. Hence, 
$$ \widehat\nu(m_1,\ldots,m_\ell) = \dfrac{1}{1-i2\pi\theta} \neq1, $$
which proves the claim.

It follows by Lemma~\ref{lemma:random walk} that the sequence 
$$\Bigl( \log_{b_1}(Y_k)\mod 1,\ldots, \log_{b_\ell}(Y_k)\mod 1 \Bigr)$$ 
is uniformly distributed in $(\RR/\ZZ)^\ell$, and the same is true if we replace $Y_k$ by $X_k$. The end of the proof goes with similar arguments as for Theorem~\ref{thm:benford}.
\end{proof}

\subsection*{Cassels-Schmidt-Benford sequences}
Given two multiplicatively independent positive integers $b_1$ and $b_2$, Bodo Volkmann defined a \emph{Cassels-Schmidt number of type $(b_1,b_2)$} as a number which is normal in base $b_1$ but not in base $b_2$. By analogy, he also proposed to define a \emph{Cassels-Schmidt-Benford (CSB) sequence of type $(b_1,b_2)$} as a sequence of positive numbers $(x_k)$ which follows Benford's law with respect to base $b_1$ but not with respect to base $b_2$.

It turns out that it is far easier to find explicit examples of CSB sequences. Indeed, take $x_k\egdef (b_2)^k$, then $(x_k)$ does certainly not follow Benford's law with respect to base $b_2$. But since $\ln b_2/\ln b_1$ is not a rational number, the sequence $\left(\log_{b_1} x_k \mod 1\right)=\left(k \ln b_2/\ln b_1 \mod 1\right)$ is uniformly distributed in $[0,1]$. Hence $(x_k)$ follows Benford's law with respect to base $b_1$.

As an application of Theorem~\ref{thm:base} in the case $\ell=2$, we get that in almost every URI-set, we can find a CSB sequence $(x_k)$ of type $(b_1,b_2)$, such that the sequence $\left(\mant_{b_2}(x_k)\right)$ follow any probability distribution prescribed in advance on $[1,b_2[$.

\section{Local URI-density}
For the sake of simplicity, we now return to the classical base~10 numeration system (but obviously the following results also hold in any integer base).

\subsection{A single element of a URI-set satisfies Benford's law}
\label{sec:Benford-local}
According to Theorem~\ref{thm:log-density}, Theorem~\ref{thm:benford} turns out to be weaker than Flehinger's theorem. However similar ideas to those developed in the proof can lead to somewhat stronger results than Theorem~\ref{thm:benford}.

\begin{theo}
 \label{thm:single}
For all $1\le \alpha<\beta <10$
$$ \lim_{k\to\infty}\PP\Bigl(\mant(N_k)\in[\alpha, \beta ]\Bigr)=\Pb([\alpha, \beta ]). $$
\end{theo}

\begin{proof}
As in the proof of Theorem~\ref{thm:benford}, we consider the URI-set $\E$ constructed from the Poisson process $\xi$. For a fixed integer $m$, we number the points of $\xi\cap ]m, +\infty[$: 
$$ \xi\cap ]m, +\infty[ = \{X_1^m<X_2^m<\ldots\}. $$
We also set $X_0^m\egdef m$.
Observe that the process $\left(1/X_k^m\right)_{n\ge0}$ is again a Markov chain such that, given $\left(1/X_0^m,\ldots,1/X_k^m\right)$, $1/X_{k+1}^m$ is uniformly distributed in $]0,1/X_k^m[$. It follows by~\eqref{eq:speed of convergence} and Remark~\ref{rem:inverse} that, for any Borel set $S\subset[1,10[$,
\begin{equation}
 \label{eq:convergence of X_k^m}
\forall k\ge 1,\quad \left| \PP\Bigl( \mant \left(X_k^m\right) \in S\Bigr)-\Pb(S) \right| \le \left(\dfrac{9}{10}\right)^k.
\end{equation}
We now consider the event 
$$ A_m\egdef \bigcap_{n>m}\left(\xi_{]n-1,n]}\le1\right)\cap\bigcap_{\ell: 10^\ell>m}\left(\xi_{]10^\ell-1,10^\ell]}=0\right). $$ 
Defining the random variable $J_m$ as the largest index such that $N_{J_m}\le m$, the realization of $A_m$ ensures that the shift of index between the process $(X_k^m)$ and the integers $N_k$ that are larger than $m+1$ remains constant, equal to $J_m$. Therefore, for any $k\ge 1$
$$ 0\le N_{J_m+k}-X_k^m <1. $$
Moreover, since $A_m$ also forbids that any $N_k>m$ be of the form $10^\ell$, we get (conditionally to $A_m$)
\begin{equation}
 \label{eq:conditionally to A_m}
\forall k\ge 1,\quad 0\le \mant\left(N_{J_m+k}\right) - \mant\left(X_k^m\right) < \dfrac{10}{N_{J_m+k}} < \dfrac{10}{m}. 
\end{equation}
Note also that $\PP(A_m)\to1$ as $m\to\infty$, so that choosing $m$ large enough will enable us to condition with respect to $A_m$ without affecting too much the probability of any event. Indeed, we will make use of the following inequality, valid for any events $A$ and $B$ with $\PP(A)>0$:
\begin{equation}
 \label{eq:conditioning}
 \left|\PP(B\,|\,A)-\PP(B)\right| \le \dfrac{\PP(A^c)}{\PP(A)}.
\end{equation}

Let us fix an arbitrary $\varepsilon>0$. We choose $m$ large enough so that 
\begin{equation}
 \label{eq:m large enough}
\dfrac{\PP(A_m^c)}{\PP(A_m)}<\varepsilon\quad \text{ and }\quad \dfrac{10}{m}<\varepsilon.
\end{equation}

Conditioning with respect to $J_m$ which takes values in $\{1,\ldots,m\}$, we get
\begin{multline*} 
 \left| \PP\Bigl(\mant(N_k)\in[\alpha, \beta ]\Bigr) - \Pb([\alpha, \beta ]) \right| \\
\le 
   \sum_{j=1}^{m}  \PP(J_m=j) \left| \PP\Bigl(\mant(N_k)\in[\alpha, \beta ]\,|\,J_m=j \Bigr) - \Pb([\alpha, \beta ]) \right|.
\end{multline*}
Then we write, for any $1\le j\le m$,
$$ \left| \PP\Bigl(\mant(N_k)\in[\alpha, \beta ]\,|\,J_m=j \Bigr) - \Pb([\alpha, \beta ]) \right| \le D_1+D_2+D_3+D_4, $$
where
\begin{align*}
 &D_1 \egdef \left| \PP\Bigl(\mant(N_k)\in[\alpha, \beta ]\,|\,J_m=j \Bigr) - \PP\Bigl(\mant(N_k)\in[\alpha, \beta ]\,|\,A_m,\ J_m=j \Bigr) \right|\\
 &D_2 \egdef \left| \PP\Bigl(\mant(N_{J_m+k-j})\in[\alpha, \beta ]\,|\,A_m,\ J_m=j \Bigr)\right.\\
 &           \hspace{5cm}      \left.- \PP\Bigl(\mant(X_{k-j}^m)\in[\alpha, \beta ]\,|\,A_m,\ J_m=j \Bigr) \right|,\\
 &D_3 \egdef  \left| \PP\Bigl(\mant(X_{k-j}^m)\in[\alpha, \beta ]\,|\,A_m,\ J_m=j \Bigr) - \PP\Bigl(\mant(X_{k-j}^m)\in[\alpha, \beta ]\,|\,J_m=j \Bigr) \right|\\
&D_4 \egdef \left| \PP\Bigl(\mant(X_{k-j}^m)\in[\alpha, \beta ]\,|\,J_m=j \Bigr) - \Pb([\alpha, \beta ]) \right|.
\end{align*}

Observe that $A_m$, which is measurable with respect to the Poisson process on $]m,+\infty[$, is independent of $(J_m=j)$, which is measurable with respect to the Poisson process on $]1,m]$. Hence, using~\eqref{eq:conditioning} and~\eqref{eq:m large enough}, we can bound $D_1+D_3$ by $2\varepsilon$. 

Again, $X_{k-j}^m$ is measurable with respect to the Poisson process on $]m,+\infty[$, hence is independent of $(J_m=j)$. By~\eqref{eq:convergence of X_k^m}, the contribution of $D_4$ can be bounded by $(9/10)^{k-j}$, hence by $(9/10)^{k-m}$.

It remains to deal with $D_2$. Since everything is conditioned on $A_m$, we can use~\eqref{eq:conditionally to A_m} to get
\begin{multline*}
 D_2 \le \PP\Bigl(\mant(X_{k-j}^m)\in[\alpha-10/m,\alpha]\,|\,A_m,\ J_m=j \Bigr) \\
+ \PP\Bigl(\mant(X_{k-j}^m)\in[\beta-10/m,\beta]\,|\,A_m,\ J_m=j \Bigr). 
\end{multline*}
Using again~\eqref{eq:conditioning}, \eqref{eq:m large enough} and~\eqref{eq:convergence of X_k^m} yields
$$ D_2 \le 2 (9/10)^{k-m} + 2\varepsilon + \Pb\left([\alpha-10/m,\alpha]\right) + \Pb\left([\beta-10/m,\beta]\right) \le 2 (9/10)^{k-m} + 4\varepsilon. $$
\end{proof}

Remark that the statement of Theorem~\ref{thm:single} would not hold if we replace the interval $[\alpha, \beta ]$ by any Borel set $S$. Indeed, since $N_k$ is an integer, the probability that its mantissa belong to the set of irrational numbers is zero. In other words, the convergence of the distribution of $\mant(N_k)$ to Benford's law is only a weak convergence. However, if we denote by $\tilde X_k$ the largest point of the Poisson process which is smaller than $N_k$, then $\tilde X_k\in]N_k-1,N_k]$, and the distribution of $\mant(\tilde X_k)$ converges to Benford's law in total variation norm.

\subsection{Local URI-density is stronger than natural density}
\label{sec:localURI-density}
In view of Theorem~\ref{thm:single}, it is natural to introduce the \emph{local URI-density} of $A\subset\ZZ_+$
as the limit, when $k\to\infty$, of $\PP(N_k\in A)$ (whenever the limit exists).
The purpose of this section is to prove that local URI-density is stronger than natural density:

\begin{theo}
\label{thm:natural_implies_local_URI}
 If $A\subset\ZZ_+$ possesses a natural density, then the local URI-density of $A$ exists and coincides with its natural density.
\end{theo}

In fact, local URI-density turns out to be \emph{strictly} stronger than natural density, since Theorem~\ref{thm:single} proves the existence of sets $A$ without natural densities but for which the local URI-density exists.

The proof of Theorem~\ref{thm:natural_implies_local_URI} is based on the following lemma.
\begin{lemma}
 \label{lemma:type_de_distributions}
Let $(P_k)_{k\ge1}$ be a sequence of probability distributions on $\ZZ_+$ satisfying: For any $k\ge1$, there exists $n_k$ such that
\begin{itemize}
 \item The map $n\mapsto P_k(n)$ is non-decreasing on $\{1, \dots, n_k\}$ and non-increasing on $\{n_k, n_k+1, \dots\}$;
 \item For any integer $m\ge1$, 
$$
\lim_{k\to\infty}\frac{P_k(mn_k)}{P_k(n_k)}=1 .
$$
\end{itemize}
Then, if $A\subset\ZZ_+$ possesses a natural density $\alpha$, $\lim_{k\to\infty} P_k(A)=\alpha$.
\end{lemma}
\begin{figure}[htp]
\label{fig:profile}
 \centering
 \input{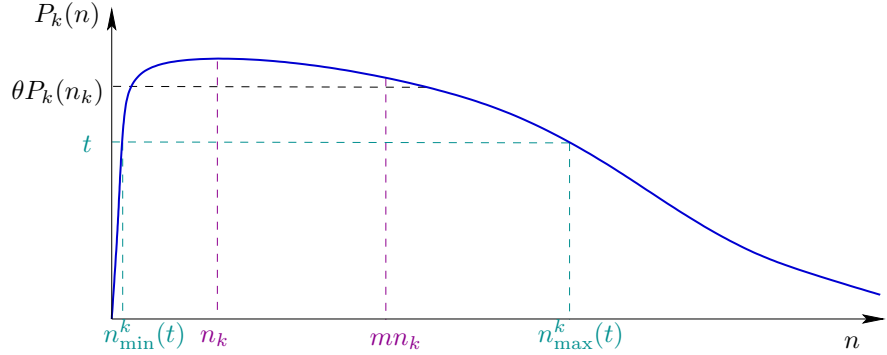}
 \caption{Profile of $P_k$}
\end{figure}

\begin{proof}
Let us fix $\varepsilon>0$. 
 Let $\theta\in]0,1[$, close enough to 1 so that $(1-\theta)/\theta<\varepsilon$.
Let $m\in\ZZ_+$ be such that $m>1/\varepsilon$ and such that, for any $n\ge m$, 
$$\frac1{n}\sum_{i=1}^n \ind{A}(i) \in  \,]\alpha-\varepsilon, \alpha+\varepsilon[.$$
We choose $k$ large enough such that 
\begin{equation}
 \label{Eq:choice of k}
\frac{P_k(mn_k)}{P_k(n_k)} > \theta.
\end{equation}
By a Fubini argument, we can write $P_k(A)$ as
\begin{equation}
 \label{Eq:Fubini}P_k(A)=\int_0^{P_k(n_k)} \left|\{n\in A: P_k(n)>t\}\right|\, dt.
\end{equation}
We split the integral into two terms 
$$
I_1\egdef \int_0^{\theta P_k(n_k)} \left|\{n\in A: P_k(n)>t\}\right|\, dt \mbox{ and }
I_2\egdef \int_{\theta P_k(n_k)}^{P_k(n_k)} \left|\{n\in A: P_k(n)>t\}\right|\, dt.
$$
Observe that 
\begin{multline*}
 \theta P_k(n_k)\ \left|\{n\in\ZZ_+: P_k(n)>\theta P_k(n_k)\}\right| \\
\le P_k\left(\left|\{n\in\ZZ_+: P_k(n)>\theta P_k(n_k)\}\right|\right) \le 1.
\end{multline*}
Therefore
$$
I_2\le (1-\theta)P_k(n_k)\ \left|\{n\in\ZZ_+: P_k(n)>\theta P_k(n_k)\}\right|\le \frac{1-\theta}{\theta}<\varepsilon.
$$
Let us turn to the estimation of $I_1$. 
By the hypothesis on the variations of $P_k(n)$, for any $0 <t<P_k(n_k)$, there exist 
$n_{\min}^k(t)\le n_k\le n_{\max}^k(t)$ such that
$$
\left\{n: P_k(n)\ge t\right\} = \left\{n_{\min}^k(t), \dots, n_{\max}^k(t)\right\}.
$$
(See Figure~\ref{fig:profile}.) We can rewrite $I_1$ as 
$$
\int_0^{\theta P_k(n_k)} \Bigl(n_{\max}^k(t)-n_{\min}^k(t)+1\Bigr) \varphi(t)\, dt, 
$$
where 
$$
\varphi(t)\egdef \frac1{n_{\max}^k(t)-n_{\min}^k(t)+1} \sum_{i=n_{\min}^k(t)}^{n_{\max}^k(t)} \ind{A}(i).
$$
We prove that, for  $0 <t<\theta P_k(n_k)$, $\varphi(t)$ is close to the natural density of $A$:
By~\eqref{Eq:choice of k}, for any $0 <t<\theta P_k(n_k)$, we have $n_{\max}^k(t)> m n_k$. Thus
$$
1\le \frac{n_{\max}^k(t)}{n_{\max}^k(t)-n_{\min}^k(t)+1} \le \frac{m}{m-1}\le \frac1{1-\varepsilon},
$$
and 
$$
\frac{1}{n_{\max}^k(t)}\sum_{n=1}^{n_{\min}^k(t)-1} \ind{A}(n) \le \frac{n_{\min}^k(t)-1}{n_{\max}^k(t)}
\le \frac{1}{m} <\varepsilon.
$$
Since $n_{\max}^k(t)>m$, 
$$
\frac{1}{n_{\max}^k(t)}\sum_{n=1}^{n_{\max}^k(t)} \ind{A}(n) \in \,]\alpha-\varepsilon, \alpha+\varepsilon[, 
$$
It follows that for $0 <t<\theta P_k(n_k)$, $ \alpha-2\varepsilon < \varphi(t) < (\alpha+\varepsilon)/(1-\varepsilon)$.

Hence we get the following estimation: 
$$ (\alpha-2\varepsilon)I_3< I_1 < \dfrac{\alpha+\varepsilon}{1-\varepsilon} I_3, $$
where
$$ I_3\egdef\int_0^{\theta P_k(n_k)} \Bigl(n_{\max}^k(t)-n_{\min}^k(t)+1\Bigr) \, dt .$$
Using \eqref{Eq:Fubini} with $A=\ZZ_+$, we get
$$ \int_0^{P_k(n_k)} \Bigl(n_{\max}^k(t)-n_{\min}^k(t)+1\Bigr) \, dt = P_k(\ZZ_+) = 1,$$
and by the same argument as for the estimation of $I_2$, we have
$$ \int_{\theta P_k(n_k)}^{P_k(n_k)} \Bigl(n_{\max}^k(t)-n_{\min}^k(t)+1\Bigr) \, dt < \frac{1-\theta}{\theta}<\varepsilon. $$
hence $1-\varepsilon<I_3\le 1$.
\end{proof}

Observe that the second condition in the lemma is crucial. Indeed, the sequence of binomial distributions of parameter $p\in]0,1[$ defined by 
$$B_k(n)\egdef \binom{k}{n} p^n(1-p)^{k-n}, 0\le n\le k$$
satisfies the first assumption of the lemma. However there exists a set $A$ possessing a natural density, but for which 
$B_k(A)$ fails to converge to this natural density as $k\to\infty$ (see~\cite{diaconis_thesis}, Theorem~3 page~25).

\begin{proof}[Proof of Theorem~\ref{thm:natural_implies_local_URI}]
 Defining $P_k(n)\egdef \PP(N_k=n)$, we have to check the hypotheses of Lemma~\ref{lemma:type_de_distributions}.
We start by establishing an induction formula for $P_k(n)$. Recall that $\PP(N_k=n)=0$ if $n<k$, and that $\PP(N_1=n)=\ind{n=1}$. For $2\le k\le n$, we decompose
$$ \PP(N_k=n) =  \PP(N_k=n, N_{k-1}=n-1) + \PP(N_k=n, N_{k-1}\le n-2). $$
The first term in the RHS is $(1/n)\PP(N_{k-1}=n-1)$. The second term can be written as
\begin{align*}
 \PP(N_k=n, N_{k-1}<n-1) &= \frac{1}{n}\left(1-\frac{1}{n-1}\right) \PP(|\E_{n-2}|=k-1)\\
& = \frac{n-2}{n}\frac{1}{n-1} \PP(|\E_{n-2}|=k-1)\\
& = \frac{n-2}{n} \PP(N_k=n-1).
\end{align*}
This yields, for any $2\le k \le n$,
\begin{equation}
 \label{eq:induction for P_k}
P_k(n) = \frac{n-2}{n} P_k(n-1) + \frac{1}{n} P_{k-1}(n-1). 
\end{equation}
For $2\le k\le n-1$, dividing by $P_k(n-1)$ we obtain
\begin{equation}
\label{eq:f_n}
 \frac{P_k(n)}{P_k(n-1)}=1+\frac{1}{n}\Bigl(f_n(k)-2\Bigr),
\end{equation}
where
$$ f_n(k)\egdef \frac{P_{k-1}(n-1)}{P_k(n-1)} 
. $$
We prove by induction on $n\ge 4$ that $k\in\{2,\ldots,n-1\}\mapsto f_n(k)$ is a non-decreasing function. Observe that $f_4(2)=0$ and $f_4(3)=1$, hence $f_4$ is non-decreasing. Assume that $f_{n-1}$ is non-decreasing for some $n\ge5$. Using~\eqref{eq:induction for P_k}, we get for $3\le k\le n-1$
$$  
f_n(k) = \dfrac{\dfrac{n-3}{n-1} P_{k-1}(n-2)+\dfrac{1}{n-1}P_{k-2}(n-2)}{\dfrac{n-3}{n-1}P_{k}(n-2)+\dfrac{1}{n-1}P_{k-1}(n-2) }.
$$
If $3\le k\le n-2$, we get 
\begin{equation}
\label{eq:f_nbis}
f_n(k) = \dfrac{n-3 + f_{n-1}(k-1)}{\dfrac{n-3}{f_{n-1}(k)}+1}. 
\end{equation}
Hence, by induction, $f_n$ is non-decreasing on $\{2, \dots, n-2\}$.
Moreover, for $k=n-1$, since $P_{n-1}(n-2)=0$,
$$ f_n(n-1) = n-3 + f_{n-1}(n-2) \ge n-3 + f_{n-1}(n-3) \ge f_n(n-2) .$$
Observe also that $f_n(2)=0$ and $f_n(n-1)=(n-3)(n-2)/2$ for all $n\ge 4$. 
Hence, for all $n\ge5$, there exists an integer $k_n$ such that $f_n(k)\le 2$ for $2\le k\le k_n$ and $f_n(k)> 2$ for $k> k_n$.
Since $f_{n-1}(k-1)\le f_{n-1}(k)$ for any $3\le k\le n-2$, we get by~\eqref{eq:f_nbis} that 
$f_n(k)\le f_{n-1}(k)$, which proves that $n\mapsto k_n$ is non-decreasing. 

For any fixed $k\ge 3$, let $n_k$ be the smallest integer $n$ such that $k_n\ge k$. 
By~\eqref{eq:f_n}, $n\mapsto P_k(n)$ is non-decreasing upto $n_k$ and non-increasing after $n_k$. 
Note that $n_k$ exists, otherwise $n\mapsto P_k(n)$ would be non-decreasing, which is obviously impossible.
Therefore, the first hypothesis of Lemma~\ref{lemma:type_de_distributions} is satisfied.

For all $k\ge3$, observe that $n_k$ is characterized by the following:
\begin{equation}
 \label{eq:n_k}
f_{n_k}(k)\le 2,\quad\mbox{and } f_{n_k-1}(k)>2.
\end{equation}

To check that $(P_k)$ satisfies the second hypothesis, we need precise estimations of $f_n(k)$. We start by establishing a formula for $P_k(n)$.
Observe that for all $1\le j <n$, $\PP(N_k=n | N_{k-1}=j)$ is equal to 
$$
\PP(j+1\notin\E, \dots n-1\notin\E, n\in\E)=\frac{j}{j+1}\dots\frac{n-2}{n-1}\frac{1}{n}=
\frac{j}{n(n-1)}.
$$
Hence, by conditioning, $P_k(n)=\PP(N_k=n)$ can be rewritten as 
\begin{multline*}
\sum_{2\le j_2< \dots <j_{k-1}\le n-1} \PP(N_k=n | N_{k-1}=j_{k-1}) \PP(N_{k-1}=j_{k-1}|N_{k-2}=j_{k-2})\\
\dots \quad\PP(N_{3}=j_{3}|N_{2}=j_{2})\PP(N_{2}=j_{2})
\end{multline*}
which yields 
\begin{align*}
P_k(n)
=& \sum_{2\le j_2< \dots <j_{k-1}\le n-1} \frac{j_{k-1}}{n(n-1)}\frac{j_{k-2}}{j_{k-1}(j_{k-1}-1)} \dots \frac{j_{2}}{j_{3}(j_{3}-1)} \frac{1}{j_{2}(j_{2}-1)} \\
=& \frac1{n(n-1)} \sum_{2\le j_2< \dots <j_{k-1}\le n-1} \frac1{j_{k-1}-1} \dots \frac1{j_{3}-1} \frac{1}{j_{2}-1}\\
=& \frac1{n(n-1)} \sum_{1\le j_2< \dots <j_{k-1}\le n-2} \frac1{j_{2}j_{3}\dots j_{k-1}}.
\end{align*}

We use this formula to estimate the denominator in the definition of $f_n(k)$:
\begin{align*}
P_k(n-1)
&=\frac1{(n-1)(n-2)} \sum_{1\le j_2< \dots <j_{k-1}\le n-3} \frac1{j_{2}j_{3}\dots j_{k-1}}\\
&=\frac1{(n-1)(n-2)} \sum_{1\le j_2< \dots <j_{k-2}\le n-3} \frac1{j_{2}j_{3}\dots j_{k-2}} \quad
C(j_2, \dots, j_{k-2}),
\end{align*}
where
$$C(j_2, \dots, j_{k-2})\egdef\frac1{k-2}\sum_{1\le j\le n-3 \atop j\notin\{j_2, \dots, j_{k-2}\}} \frac1{j}.$$
Observe that 
$$
\frac1{k-2}\sum_{k-2\le j\le n-3} \frac1{j} \le C(j_2, \dots, j_{k-2})\le \frac1{k-2}\sum_{1\le j\le n-3} \frac1{j},
$$
which gives the following estimation 
\begin{equation}
 \label{eq:estimation f_n(k)}
\frac{k-2}{\sum_{1\le j\le n-3} \frac1{j}}
\le f_n(k)=\frac{P_{k-1}(n-1)}{P_k(n-1)} 
\le \frac{k-2}{\sum_{k-2\le j\le n-3} \frac1{j}}\cdot
\end{equation}

Let $m\ge1$. Applying this estimation to $f_{n_k-1}$ and $f_{mn_k}$, we get
$$
0 \le f_{n_k-1}(k)- f_{mn_k}(k) \le \frac{k-2}{\sum_{k-2\le j\le n_k-4} \frac1{j}}- \frac{k-2}{\sum_{1\le j\le mn_k-3} \frac1{j}}\cdot
$$
The RHS of the above inequality can be written as a product $AB$, where
$$
A\egdef \frac{k-2}{\sum_{k-2\le j\le n_k-4} \frac1{j}} \quad\mbox{and }B\egdef\frac{\sum_{1\le j\le mn_k-3} \frac1{j} - \sum_{k-2\le j\le n_k-4} \frac1{j}}{\sum_{1\le j\le mn_k-3} \frac1{j}}\cdot
$$
By~\eqref{eq:estimation f_n(k)} and~\eqref{eq:n_k}, we get 
\begin{equation}
 \label{eq:lien k n_k}
\frac{k-2}{\sum_{1\le j\le n_k-3} \frac1{j}}\le2 , 
\end{equation}
which ensures that $A$ is bounded (say, by 4). Moreover, an easy computation shows that $B\sim \frac{\ln k}{\ln n_k}$, which goes to 0 as $k\to\infty$ by~\eqref{eq:lien k n_k}.
Recalling that $f_{n_k-1}(k)>2$ by~\eqref{eq:n_k}, the above estimations prove the following property: For any $\varepsilon>0$, for $k$ large enough, $f_n(k)>2-\varepsilon$ for each $n_k\le n\le mn_k$. 
For such $k$, we get by~\eqref{eq:f_n} 
$$
\frac{P_k(mn_k)}{P_k(n_k)} = \prod_{n=n_k+1}^{mn_k}\frac{P_k(n)}{P_k(n-1)}  = \prod_{n=n_k+1}^{mn_k} \left( 1+ \frac{1}{n}(f_{n}(k)-2)\right)
\ge \left(1-\frac{\varepsilon}{n_k}\right)^{(m-1)n_k}.
$$
On the other hand, we know that $P_k(n_k)\ge P_k(mn_k)$, which proves that
$$
\lim_{k\to\infty}\frac{P_k(mn_k)}{P_k(n_k)}=1 .
$$
\end{proof}

\section{Open problems and discussion}
\label{sec:discussion}
\subsection{Connection with other densities}

We conjecture that the existence of local URI-density implies the existence of URI-density (and, in this case, that both coincide).

It is not clear either whether local URI-density is equivalent to the $H^\infty$-density used by Flehinger. 
However we can provide the following interpretation of the $H^\infty$-density in terms of our URI-set. Recall that in the proof of Theorem~\ref{thm:cesaro}, we ordered the elements of $\E_n=\E\cap\{1,\ldots,n\}$ backwards:
$$ \E_n=\left\{Y_1^{(n)} > Y_2^{(n)} > \ldots > Y_{|\E_n|}^{(n)}=1\right\}.$$

\begin{prop}[Stochastic interpretation of $H^\infty$-density]
 For $A\subset\ZZ_+$, $A$ has $H^\infty$-density $\alpha$ if and only if 
$$ \lim_{k\to\infty} \liminf_{n\to\infty} \PP\left(Y_k^{(n)}\in A\right)
=\lim_{k\to\infty} \limsup_{n\to\infty} \PP\left(Y_k^{(n)}\in A\right)=\alpha. $$
\end{prop}

\begin{proof}
For each $n\ge1$, we introduce a non-increasing sequence of random integers $(\widetilde Y_i^{(n)})_{i\ge1}$, with the following distribution:
\begin{itemize}                                                                                                                                       \item $\widetilde Y_1^{(n)}$ is uniformly distributed in $\{1,\ldots,n\}$,
\item conditionally to $\widetilde Y_1^{(n)},\ldots,\widetilde Y_i^{(n)}$, the random variable $\widetilde Y_{i+1}^{(n)}$ is uniformly distributed in $\{1,\ldots,\widetilde Y_i^{(n)}\}$.                                                                                                                                      \end{itemize}
For $A\subset\ZZ_+$, we write $\PP\left(\widetilde Y_k^{(n)}\in A\right)$ as
\begin{multline*} 
\sum_{1\le y_k\le y_{k-1}\le\dots\le y_1\le n} \ind{A}(y_k) \PP\left(\widetilde Y_k^{(n)}=y_k|\widetilde Y_{k-1}^{(n)}=y_{k-1}\right) \cdots\\ \PP\left(\widetilde Y_2^{(n)}=y_2|\widetilde Y_{1}^{(n)}=y_{1}\right) \PP\left(\widetilde Y_1^{(n)}=y_1\right),
\end{multline*}
which yields
$$ \PP\left(\widetilde Y_k^{(n)}\in A\right) = \frac{1}{n}\sum_{y_1=1}^n\frac{1}{y_1}\sum_{y_2=1}^{y_1} \dots \frac{1}{y_{k-1}}\sum_{y_k=1}^{y_{k-1}}\ind{A}(y_k). $$
We recognize $P_n^k$ used in the definition of the $H^\infty$-density (see Section~\ref{sec:log}).
Now, we observe that 
$$ \PP\left(Y_k^{(n)}\in A\right) = \PP\left(\widetilde Y_k^{(n)}\in A | D_k^n\right),  
\mbox{ where }
D_k^n \egdef \left\{\widetilde Y_1^{(n)}>\ldots >\widetilde Y_k^{(n)}\right\}.$$
It remains to prove that, for any fixed $k$, $\PP(D_k^n)\rightarrow 1$ as $n\rightarrow \infty$.
Fix $\epsilon>0$ and choose $\delta>0$ such that $(1-3\delta)^k>1-\epsilon$.
Observe that whenever $n>1/\delta$, the proportion of integers $i\in\{1, \dots, n\}$ such that $\delta<i/n<1-\delta$ is larger than $1-3\delta$. 
Now, if $n>1/\delta^{k}$, we have 
$$
\PP\left(\frac{\widetilde Y_1^{(n)}}{n}\in ]\delta, 1-\delta[,\, \frac{\widetilde Y_2^{(n)}}{\widetilde Y_1^{(n)}}\in ]\delta, 1-\delta[,\, \dots,\,  \frac{\widetilde Y_k^{(n)}}{\widetilde Y_{k-1}^{(n)}}\in ]\delta, 1-\delta[\right) > (1-3\delta)^k>1-\epsilon.
$$
Hence, $\PP(D_k^n)>1-\epsilon$.
\end{proof}

\subsection{Conditional URI-density}

Let $P$ be a subset of $\ZZ_+$ with $\sum_{p\in P}1/p=\infty$, so that the cardinality of $P\cap\E$ be almost surely infinite. We have two ways to define the \emph{URI-density of $A$ conditioned on $P$}. First, by averaging over $P\cap\E$, and consider (whenever it exists) the almost-sure limit of
$$ \dfrac{\sum_{k=1}^n \ind{P\cap A}(N_k)}{\sum_{k=1}^n \ind{P}(N_k)}. $$
Second, by numbering the elements of $P=\{p_1<p_2<\cdots<p_n<\cdots\}$ and averaging over the random subset of $P$
$$ \{p_{N_1}<p_{N_2}<\cdots<p_{n_k}<\cdots\}, $$
that is by considering (whenever it exists) the almost-sure limit of
$$ \dfrac{1}{n} \sum_{k=1}^n \ind{A}(p_{N_k}). $$
Question: are these two definitions equivalent?

\subsection{Asymptotic independence of successive elements in $\E$}

Another question concerning the URI-set $\E$ is the following: consider $A,B\subset \ZZ_+$, and assume that for both $A$ and $B$ we can define the density $d(A)$ and $d(B)$ (these could be the natural densities, the URI-densities, the $H_\infty$-densities or maybe some other notions of densities). 
Under which condition do we have
$$ \dfrac{1}{n} \sum_{k=1}^n \ind{A}(N_k)\ind{B}(N_{k+1}) \tend[\text{a.s.}]{n}{\infty} d(A) d(B) \ ?$$

We conjecture that it is true when both $d(A)$ and $d(B)$ are natural densities. However this is certainly not true for all $A$ and $B$ with URI-densities: As a counterexample, consider the set $A$ of integers with leading digit 1 and the set $B$ of integers with leading digit 9. 

But what happens if for example $d(B)$ is the natural density whereas $d(A)$ is only the URI-density?

\subsection{Iterated URI-density}

Diaconis defined in~\cite{diaconis_thesis} the \emph{iterated log-density}: For a subset $A$ of $\ZZ_+$, set
$$ L(A,n,1)\egdef \dfrac{1}{\ln n} \sum_{j=1}^n \dfrac{1}{j}\ind{A}(j), $$
and inductively for all $\ell\ge2$:
$$ L(A,n,\ell)\egdef \dfrac{1}{\ln n} \sum_{j=1}^n \dfrac{1}{j} L(A,j,\ell-1). $$
The set $A$ has \emph{$\ell$-th log-density} $\alpha$ if the limit of the above exists as $n\to\infty$ and is equal to $\alpha$. In fact, this notion does not yield a new density, since Diaconis proved that $A$ has an $\ell$-th log-density if and only if $A$ has a log-density (and then, of course, both coincide). Then he proposed to define the \emph{$L_\infty$-density}, which extends the log-density in much the same way as $H_\infty$-density extends natural density: Consider $\lim_{\ell\to\infty}\liminf_{n\to\infty}L(A,n,\ell)$ and $\lim_{\ell\to\infty}\limsup_{n\to\infty}L(A,n,\ell)$. If the two limits are equal, call their common value the $L_\infty$-density of $A$. As shown in~\cite{diaconis2}, $L_\infty$-density is strictly stronger than log-density. 

It is natural in this context to study iterations of the URI-density, which can be defined as follows: Let $\left(\E^{(\ell)}\right)_{\ell\ge1}$ be a sequence of independent URI-sets, and denote the random elements of $\E^{(\ell)}$ by $N^{(\ell)}_1=1<N^{(\ell)}_2<\cdots<N^{(\ell)}_k<\cdots$. We say that $A\subset\ZZ_+$ has \emph{URI-density of order 2} equal to $\alpha$ if 
$$ \dfrac{1}{n} \sum_{k=1}^n\ind{A}\left(N^{(1)}_{N^{(2)}_k}\right) \tend[a.s.]{n}{\infty} \alpha. $$
We define the \emph{URI-density of order $\ell$} in the same way, averaging along the subsequence 
$N^{(1)}_{N^{(2)}_{\ddots_{N^{(\ell)}_k}}}$.

We can also introduce the infinite iteration of the URI-method, considering almost-sure limsup and liminf in the above expressions, and see if they converge to the same limit as $\ell\to\infty$.

Although we have shown that URI-density and log-density coincide, it is not obvious if there are connections between iterated URI-density and iterated log-density. Can URI-densities of finite order $\ell$ be strictly stronger than URI-density? Can we compare the infinite iteration of both methods?

\bibliography{uri}
\end{document}